\documentclass{amsart}

\usepackage{amssymb,amsmath,amsthm,verbatim}
\usepackage{bm} % bold inside math equns
\usepackage{tikz} %tikz figures
\usepackage{float} %figure positions
\usetikzlibrary{shapes,snakes}

\usepackage{verbatim} % A good package to expose codes in your work
\usetikzlibrary{calc}
\allowdisplaybreaks %this allows equation breaks into pages
\usepackage[colorlinks,citecolor=red,urlcolor=blue,bookmarks=false,hypertexnames=true]{hyperref} 

\newtheorem{theorem}{Theorem}
\newtheorem{corollary}{Corollary}
\newtheorem{proposition}{Proposition}
\newtheorem{lemma}{Lemma}
\theoremstyle{definition}

\theoremstyle{remark}

\numberwithin{equation}{section}

\newcommand{\ao}[1]{\tilde{#1}_{\omega}}
\newcommand{\aso}[1]{\tilde{#1}_{\sigma \omega}}
\newcommand{\as}[2]{\tilde{#1}_{\sigma^{#2} \omega}}
\newcommand{\aoa}[1]{{\tilde{#1}_{\omega}}^{\alpha}}
\newcommand{\aob}[1]{{\tilde{#1}_{\omega}}^{\beta_{ij}}}
\newcommand{\ttt}[1]{\tilde{#1}}

\newcommand{\aaa}[1]{{#1}^{\alpha}}
\newcommand{\bbb}[1]{{#1}^{\beta_{ij}}}
\newcommand{\bbbb}[1]{{#1}^{\beta}}
\newcommand{\at}[2]{\tilde{#1}_{#2}}

\begin{document}

\author{Supun T. Samarakoon}
\title{Generalized Grigorchuk's Overgroups as points on $\mathcal{M}_k$}

\maketitle

\begin{abstract}
	Following the construction from \cite{Gri84} we generalize the Grigorchuk's overgroup $\ttt{\mathcal{G}}$, studied in \cite{BG02} to the family $\{ \ao{G}, \omega \in \Omega = \{ 0, 1, 2 \}^\mathbb{N} \}$ of generalized Grigorchuk's overgroups. We consider these groups as 8-generated and describe the closure of this family in the space $\mathcal{M}_8$ of marked groups.
\end{abstract}

% Begin Section
%The title of your section 1
\section{Introduction}

Grigorchuk's space $\mathcal{M}_k$ of marked groups with $k (\geq 2)$ generators was introduced in 1984 \cite{Gri84}. It is a totally disconnected, compact metric space with complicated structure of isolated points as shown by Y. de Cornulier, L. Guyot and W. Pitsch \cite{CGP07} and non-trivial perfect kernel $\mathcal{K}$ that is homeomorphic to a Cantor set. The space also was studied in \cite{Cha00,CG05,BK19} and other articles.

The space of marked groups was used by Grigorchuck to show that his family $\{ G_\omega \}_{\omega \in \{0,1,2\}^\mathbb{N}}$ of groups of intermediate growth (between polynomial and exponential) constructed in \cite{Gri84} consist of infinitely presented groups (when $\omega$ is not virtually constant). Also, a modification of the construction lead him to show in \cite{Gri84}, that the family is closed and perfect subset of $\mathcal{M}_4$ and hence is homeomorphic to a Cantor set.

The further investigations showed usefulness of spaces $\mathcal{M}_k, k \geq 2$ for study of group properties such as (non-elementary) amenability and for constructions in group theory, in particular to study of IRS (invariant random subgroups) on a free group and other groups \cite{Bow15,BGN15}. %\textcolor{blue}{The methods introduced in \cite{Gri84} and this article could be used to find new examples of non-elementary amenable groups. (More on this after corollary \ref{infinitely presented}.)}

Let $\Omega_2 \subset \Omega = \{0,1,2\}^\mathbb{N}$ be the set consisting of virtually constant sequences. If $\omega \in \Omega \setminus \Omega_2$, then $G_\omega $ has intermediate growth (between polynomial and exponential growth) as shown in \cite{Gri84}. In \cite{Gri84} it was shown that the closure of the set $\mathcal{Z} = \{G_\omega | \omega \in \Omega \setminus \Omega_2 \}$ in $\mathcal{M}_4$, denoted by $\overline{\mathcal{Z}}$, is a closed set without isolated points (hence homeomorphic to a Cantor set) and $\overline{\mathcal{Z}} \setminus \mathcal{Z}$ is a countable set consisting of virtually metabelian groups, one such group $G_\omega^\alpha$ (defined using an algorithm $\alpha$ for the word problem) for each $\omega \in \Omega_2$. So,
\[ \overline{\mathcal{Z}} = \mathcal{Z} \cup \{ G_\omega^\alpha | \omega \in \Omega_2 \} = \text{Cantor set.} \]

In \cite{BG02}, Bartholdi and Grigorchuk investigated the group $\ttt{\mathcal{G}}$ (known as the Grigorchuk's overgroup) whose definition is similar to the first Grigorchuk group $\mathcal{G} = G_{(012)^\infty}$. It contains $\mathcal{G}$, fail to be torsion (in contrast with $\mathcal{G}$), but has intermediate growth, much larger than $\mathcal{G}$ and share with $\mathcal{G}$ many properties (like to be branch, self-similar, just infinite, etc). The group $\ttt{\mathcal{G}}$ is important, in particular because as is shown by Y. Vorobets (private communication), it constitute a big part of the topological full group $[[(\Lambda,T)]]$ associated with substitutional dynamical system $(\Lambda,T)$ generated by the Lysenok's substitution $\sigma: a \mapsto aca, \quad b \mapsto d, \quad c \mapsto b, \quad d \mapsto c$, where $T$ denotes the shift map in the space $\Lambda = \{a,b,c,d\}^{\mathbb{Z}}$.

In this article we, analogously to \cite{Gri84}, introduce a family $\{ \ao{G} | \omega \in \Omega \}$ of generalized overgroups and describe the structure of the closure of the set $\mathcal{X} = \{ \ao{G} | \omega \in \Omega \}$ in $\mathcal{M}_8$, which happen to be much more complicated than in the case of classical Grigorchuk groups (see figure \ref{fig:structure}). 

\begin{figure}[t!]
	\centering
	\begin{tikzpicture}
	\draw (-0.5,-1) node(top)[rectangle,minimum height=7cm,minimum width=11cm,draw,dashed]{};
	\draw (0,0) node(top)[ellipse,minimum height=2cm,minimum width=7cm,draw]{};
	\draw (node cs:name=top,anchor=90) -- (0,0);
	\draw (node cs:name=top,anchor=188) -- (0,0);
	\draw (node cs:name=top,anchor=-8) -- (0,0);
	\draw [dashed,opacity=.2] (-3.5,0) -- (-3.5,-3);
	\draw [dashed,opacity=.2] (-3.5,-3) arc (180:360:3.5 and 1);
	\draw [dashed,opacity=.2] (3.5,-3) -- (3.5,0);  
	\draw (-.5,0) ellipse (5 and 1.2);
	\draw (-5.5,-1.5) arc (180:360:5 and 1.2);
	\draw [dashed,opacity=.08] (4.5,-3) arc (0:180:5 and 1.2);
	\draw (-5.5,-3) arc (180:360:5 and 1.2);
	\draw (-5.5,-3) -- (-5.5,0);
	\draw (4.5,-3) -- (4.5,0);
	\node(X0)[label=center:{$\mathcal{X}_0$}] at (-2,.4) {};
	\node(X1)[label=center:{$\mathcal{X}_1$}] at (-4.5,-1.5) {};
	\node(X2)[label=center:{$\mathcal{X}_2$}] at (2.5,2) {};
	\node(X1a)[label=center:{$\aaa{\mathcal{X}}_1$}] at (0,-.5) {};
	\node(X2a)[label=center:{$\aaa{\mathcal{X}}_2$}] at (2,.4) {};
	\node(X2b)[label=center:{$\bbbb{\mathcal{X}}_2$}] at (-4.5,-3) {};
	\filldraw[fill,even odd rule,opacity=.1]
	(-.5-5.5,-1-3.5) rectangle (-.5+5.5,-1+3.5)
	(-5.5,-1.5) arc (180:360:5 and 1.2) -- (4.5,-3) -- (4.5,-3) arc (0:-180:5 and 1.2) -- cycle
	(0,0) ellipse (3.5 and 1)
	(node cs:name=top,anchor=90) arc (90:206:3.5 and 1) -- (0,0);
	\end{tikzpicture}
	\caption{Structure of topological closure of $\mathcal{X} = \mathcal{X}_0 \cup \mathcal{X}_1 \cup \mathcal{X}_2$ in $\mathcal{M}_8$}
	\label{fig:structure}
\end{figure}
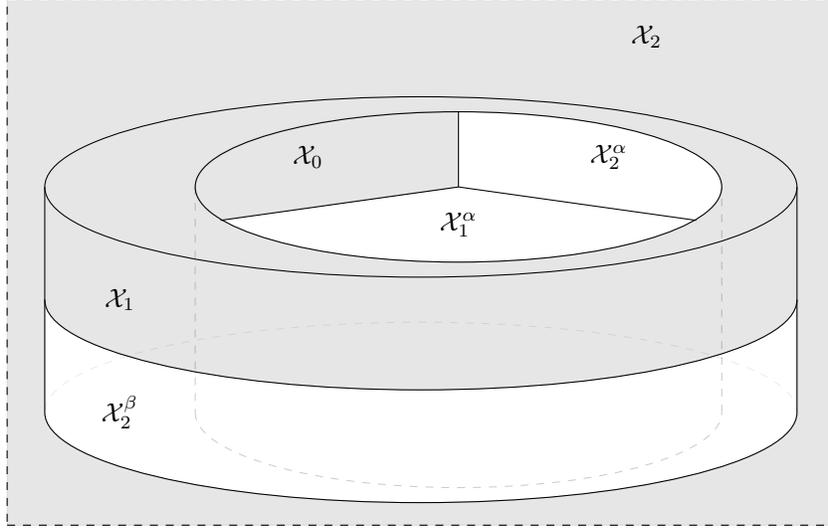

In 1957, M. Day asked whether all amenable groups are elementary amenable \cite{Day57}. It was answered negatively, by the construction of groups of intermediate growth \cite{Gri84}. Next examples came from theory of self-similar groups. One such group is the Basilica group \cite{GZ02}, which is amenable but not sub-exponentially amenable \cite{BV05}. Most recent examples are topological full groups associated with minimal Cantor system, which were used to construct finitely generated simple non-elementary amenable groups \cite{JM13}.

In 1996, Stepin observed that constructions similar to the one in \cite{Gri84}, can lead to new families of non-elementary amenable groups \cite{Ste96}. Namely, if one finds suitable Cantor set of groups containing a countable dense subset of (perhaps elementary) amenable groups and a co-meager set consisting of non-elementary groups, then standard argument based on Baire category insure that there is a co-meager set of non-elementary amenable groups. (See \cite{WW17} for non-constructive proof of existence of non-elementary amenable groups using set theoretic approach.) 

Constructions in this article are based on algorithms $\alpha$ and $\beta_{ij}$ for $i,j \in \{0,1,2\}$, which will be defined in section \ref{algo}. The algorithm $\alpha$ is a branch type algorithm, similar to the one introduced in \cite{Gri84}. Algorithms $\beta_{ij}$ were introduced in order to construct `new' class of modified overgroups (see section \ref{modified}). We hope that the methods introduced here will contribute to the study in the direction of constructing new example of non-elementary amenable groups.

Let $\Omega_0, \Omega_1 \subset \Omega$, where $\Omega_0$ is the set of all sequences with all three symbols occurring infinitely often and $\Omega_1 = \Omega \setminus (\Omega_0 \cup \Omega_2)$ is the set of all sequences with exactly two symbols occurring infinitely often. We use the word `oracle' to represent a sequence in $\Omega$.

Using algorithms $\alpha$ and $\beta_{ij}$ for $i,j \in \{0,1,2\}$, we define modified overgroups $\ao{G}^\alpha$ and $\ao{G}^{\beta_{ij}}$ (see section \ref{modified}) as those for which the word problem is decidable by the corresponding algorithm, assuming that the oracle $\omega$ is known. We define following subsets of $\mathcal{M}_8$:
\begin{align}\label{equation:sets}
	\mathcal{X}  = &\{ (\ao{G},\ao{S}) \}_{\omega \in \Omega} \text{ ; shaded region in figure \ref{fig:structure}} \nonumber \\ 
	\mathcal{X}_i  = &\{ (\ao{G},\ao{S}) \}_{\omega \in \Omega_i} \text{ ; for } i = 0, 1, 2 \nonumber\\
	\mathcal{X}_i^\alpha = &\{ (\ao{G},\ao{S}) \}_{\omega \in \Omega_i} \text{ ; for } i = 1, 2 \\
	\mathcal{X}_2^\beta  = &\{ (\ao{G}^\beta,\ao{S}^\beta) \; | \; \beta \in \{ \beta_{01}, \beta_{12}, \beta_{20} \} \; \}_{\omega \in \Omega_2}  \nonumber\\
	\mathcal{Y}  = &\mathcal{X}_0 \cup \mathcal{X}_1^\alpha \cup \mathcal{X}_2^\alpha \text{ ; middle cylinder in figure \ref{fig:structure}}  \nonumber
\end{align}

In the following text, the topological closure and the set of limit points of a set $V$ will be denoted by $\overline{V},V_\sharp$, respectively.

\begin{theorem}\label{distinct-main}
	The sets $\mathcal{X}_0, \mathcal{X}_1, \mathcal{X}_2, \aaa{\mathcal{X}}_1, \aaa{\mathcal{X}}_2,$ and $\bbbb{\mathcal{X}}_2$ are mutually disjoint subsets of $\mathcal{M}_8$. In any set other than $\bbbb{\mathcal{X}}_2$, different corresponding oracles $\omega$ give rise to different groups. In $\bbbb{\mathcal{X}}_2$, there are two different groups for each corresponding oracle $\omega$.
\end{theorem}

\begin{theorem}\label{closure}
	.
	\begin{enumerate}
		\item $\overline{\mathcal{X}} = \mathcal{X}_\sharp \sqcup \mathcal{X}_2$, where the set $\mathcal{X}_2$ consists of the set of isolated points of $\mathcal{X}$.
		\item ${\mathcal{X}}_\sharp, \mathcal{Y}$ are homeomorphic to a Cantor set. 
		\item Furthermore we have following relations: 
		\begin{enumerate}
			\item $\mathcal{Y} = (\mathcal{X}_0)_\sharp = (\aaa{\mathcal{X}}_1)_\sharp = (\aaa{\mathcal{X}}_2)_\sharp$.
			\item $\mathcal{X}_\sharp = \mathcal{Y} \cup \mathcal{X}_1 \cup \bbbb{\mathcal{X}}_2 = (\mathcal{X}_1)_\sharp = (\bbbb{\mathcal{X}}_2)_\sharp = (\mathcal{X}_2)_\sharp$.
		\end{enumerate}
	\end{enumerate}
\end{theorem}

It is worth to mention that the limit groups that appear in \cite{Gri84} are of the lamplighter type and one of them (``building block'') is a 2-extension of the lamplighter group $\mathcal{L} = \mathbb{Z}_2 \wr \mathbb{Z}$ \cite{BG14}. In our situation the lamplighter group also plays an important role and the ``building blocks'' constitute the group $\mathcal{L}$ as well as $\mathcal{L}_3 := \mathbb{Z}_2^3 \wr \mathbb{Z}$ and their direct products.

\begin{theorem}\label{modified-main}
	Let $\{ i,j,k \} = \{0,1,2\} $.
	\begin{enumerate}
		\item Let $\omega \in \Omega_2$ and let $N$ be the smallest index such that only $i$ appear after $N$. Then $\ao{G}^\alpha$ is commensurable to $(\ttt{G}_{i^\infty}^\alpha)^{2^N}$, which is virtually $(\mathcal{L}_3)^{2^N}$. Therefore $\ao{G}^\alpha$ is elementary amenable and of exponential growth. \label{alpha-omega_2}
		\item Let $\omega \in \Omega_2$ and let $N$ be the smallest index such that only $i$ appear after $N$. Then $\ao{G}^{\beta_{ij}}$ is commensurable to $(\ttt{G}_{i^\infty}^{\beta_{ij}})^{2^N}$, which is virtually $(\mathcal{L})^{2^N}$.Therefore $\ao{G}^{\beta_{ij}}$ is elementary amenable and of exponential growth. \label{beta-omega_2}
		\item Let $\omega \in \Omega_1$ and let $N$ be the smallest index such that no $k$ appear after $N$. Then $\ao{G}^\alpha$ is commensurable to $(\as{G}{N}^\alpha)^{2^N}$. $\as{G}{N}^\alpha$ contains $\mathcal{L}$ as a subgroup and is an extension of a non elementary amenable group by an abelian group. Therefore $\ao{G}^\alpha$ is non elementary amenable and of exponential growth. \label{alpha-omega_1}
	\end{enumerate}
\end{theorem}

It is known that the groups in $\mathcal{X}_2$ have polynomial growth and the groups in $\mathcal{X}_0$ and $\mathcal{X}_1$ have intermediate growth  (see \cite{Sam19}). As a consequence of theorem \ref{modified-main}, we have;

\begin{corollary}\label{growth-main}
	Groups in the set $\mathcal{X}_0 \cup \mathcal{X}_1$ are of intermediate growth, groups in the set $\mathcal{X}_2$ are of polynomial growth, and groups in $\aaa{\mathcal{X}}_1 \cup \aaa{\mathcal{X}}_2 \cup \bbbb{\mathcal{X}}_2$ are of exponential growth.
\end{corollary}

If $G$ is a finitely presented group in $\mathcal{M}_k$ with finite set of relations $R$, such that $G_n \Rightarrow G$, then $G$ maps onto $G_n$ for sufficiently large $n$. This can be obtained by considering the ball of radius $n$ centered at identity of the Cayley graph of $G$, where $n/2$ is larger than the maximum of lengths of relations in $R$. In particular, the growth rate of $G$ is not less than the growth growth rate of $G_n$. By theorem \ref{closure}, for $\omega$ non virtually constant, there is a sequence $\{ G_n \}$ of groups of exponential growth (by corollary \ref{growth-main}) in $\mathcal{X}_2^\beta$ converging to $\ao{G}$. Therefor by above argument, we get following corollary:

\begin{corollary}\label{infinitely presented}
	$\ao{G}$ is infinitely presented for $\omega \in \Omega \setminus \Omega_2$.	
\end{corollary}

%\textcolor{red}{	Also tell about lamplighter group and direct products. And L3 and products. 	} \\ \\ \\ \\
%The sets $\mathcal{X}_\sharp, \mathcal{Y}$ are Cantor spaces that contain non-elementary amenable groups and having countable dense subsets of elementary amenable groups. This is yet another example of existance of non-elementary amenable groups shown in \cite{Ste96} and \cite{WW17}. 

The Cantor-Bendixson rank is an invariant of topological spaces. It is the least ordinal at which the removal of isolated points makes no change to the space. If the topological space is Polish (complete, metrizable and separable), then the  Cantor-Bendixson rank is countable \cite{Kec95}. As a consequence of theorem \ref{closure}, the Cantor-Bendixson rank of $\overline{\mathcal{X}}$ is 1 .

% Begin Section
\section{Preliminaries}

We will be using following notations; $\Omega$, sequences of three symbols $0,1,2$, and $\Omega_0, \Omega_1, \Omega_2$ subsets of $\Omega$, where $\Omega_0$ the set of all sequences with all three symbols occurring infinitely often, $\Omega_1$ the set of all sequences with exactly two symbols occurring infinitely often, and $\Omega_2$ the set of all eventually constant sequences. Also let $\sigma : \Omega \rightarrow \Omega$ be the left shift. i.e. $(\sigma\omega)_n = \omega_{n+1}$.

% Begin Subsection
\subsection{Generalized Grigorchuk's Groups $\mathbf{G_\omega}$ and Generalized Grigorchuk's Overgroups $\mathbf{\ao{G}}$} \label{generalized}

Consider the labeled  binary rooted tree $T_2$ (see figure ~\ref{fig:tree}). For each vertex $v$, let $I$ be the trivial action on $v$ and let $P$ be the action of interchanging the vertices $v0,v1$ and acting trivially on these two vertices. We identify an infinite sequence $\{a_n\}$ of $P,I$ with the element $g \in Aut(T_2)$ such that $g \cdot (1^{n-1}0) = a_n$. We define $a$ to be the element acting on the root as $P$ and trivially on other vertices and $x$ to be the element $(P,P, \hdots)$.

% Node styles
\tikzset{
	% Two node styles for game trees: solid and hollow
	s/.style={circle,draw,inner sep=0,fill=black},
	l/.style={circle,draw,inner sep=0,fill=black,xshift=10},
	r/.style={circle,draw,inner sep=0,fill=black,xshift=-10},
	%    t/.style={regular polygon, regular polygon sides=3,draw,inner sep=2,yshift=3}
}
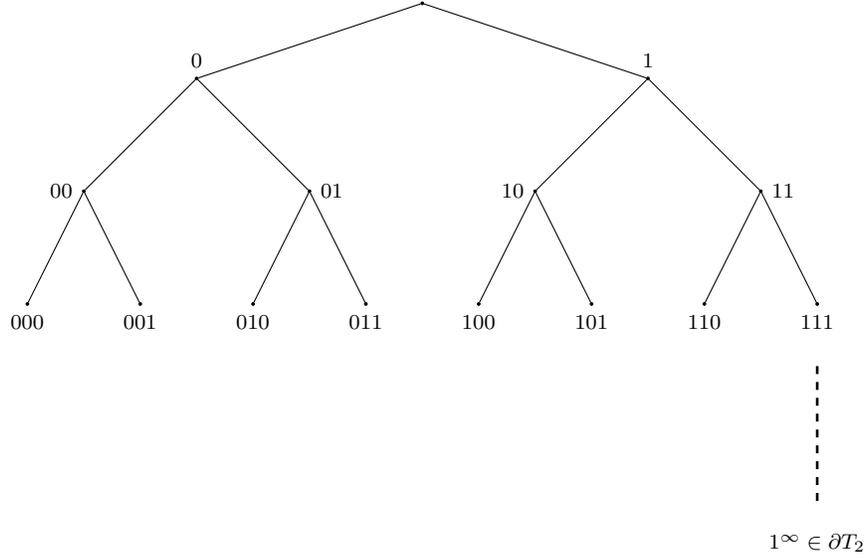
\begin{figure}[t!]
	\centering
	\begin{tikzpicture}[scale=1,font=\footnotesize]
	% Specify spacing for each level of the tree
	\tikzstyle{level 1}=[level distance=10mm,sibling distance=60mm]
	\tikzstyle{level 2}=[level distance=15mm,sibling distance=30mm]
	\tikzstyle{level 3}=[level distance=15mm,sibling distance=15mm]
	\tikzstyle{level 4}=[level distance=10mm,sibling distance=10mm]
	% The Tree
	\node(-1)[s,label=above:{}]{}
	child{node(0)[s,label=above:{$0$}]{}
		child{node(00)[s,label=left:{${00}$}]{}
			child{node(100)[s,label=below:{${000}$}]{}}
			child{node(101)[s,label=below:{${001}$}]{}}
		}
		child{node(01)[s,label=right:{${01}$}]{}
			child{node(110)[s,label=below:{${010}$}]{}}
			child{node(111)[s,label=below:{${011}$}]{}}
		}
	}
	child{node(1)[s,label=above:{$1$}]{}
		child{node(10)[s,label=left:{${10}$}]{}
			child{node(100)[s,label=below:{${100}$}]{}}
			child{node(101)[s,label=below:{${101}$}]{}}
		}
		child{node(11)[s,label=right:{${11}$}]{}
			child{node(110)[s,label=below:{${110}$}]{}}
			child{node(111)[s,label=below    :{${111}$}]{}}
		}
	};
	% information set
	%\draw[dashed,rounded corners=10]($(1) + (-.2,.25)$)rectangle($(2) +(.2,-.25)$);
	% Subgame
	%\draw[dotted, thick, red, rounded corners=15]($(4) + (-.4,1.75)$)rectangle($(5) +(.4,-.5)$);
	% specify mover at 2nd information set
	\node(1111)[yshift=-20] at ($(111)$) {};
	%    \node[s,xshift=15,yshift=-40] at ($(111)$) {};
	%    \node[s,xshift=22.5,yshift=-60] at ($(111)$) {};
	\node(11111)[yshift=-80,label=below:{$1^\infty \in \partial T_2$}] at ($(111)$) {};
	\draw[dashed,line width=1 pt] (1111) to (11111);
	%    \node[s,xshift=10,yshift=-10] at ($(111)$) {};
	\end{tikzpicture}
	\caption{Labeled binary rooted tree $T_2$}
	\label{fig:tree}
\end{figure}

For $\omega \in \Omega$, identify elements $b_\omega, c_\omega, d_\omega$ with sequences $\{b_n\}, \{c_n\}, \{d_n\}$, respectively, where
\\
$ b_n = \begin{cases} 
P & \omega_n = 0 \text{ or } 1 \\
I & \omega_n = 2 
\end{cases} \quad , \quad
c_n = \begin{cases} 
P & \omega_n = 0 \text{ or } 2 \\
I & \omega_n = 1
\end{cases} \quad , \quad
d_n = \begin{cases} 
P & \omega_n = 1 \text{ or } 2 \\
I & \omega_n = 0
\end{cases}
$.
\\
Further define $\ao{b} := xb_\omega, \; \ao{c} := xc_\omega$ and $\ao{d} := xd_\omega$. Note that all these elements are involutions and all except $a$ commute with each other. The generalized Grigorchuk's group $G_\omega$, introduced in \cite{Gri84}, is the group generated by elements $a,b_\omega, c_\omega, d_\omega$ and the generalized overgroup $\ao{G}$, is the group generated by $a,b_\omega, c_\omega, d_\omega, x$. $G_\omega \subset \ao{G}$ and it is useful to view $\ao{G}$ as the group generated by elements $a,b_\omega, c_\omega, d_\omega, x, \ao{b}, \ao{c}, \ao{d}$, where a typical element $g \in \ao{G}$ can be represented in reduced form $(a)*a*a \hdots a*a*(a)$ where first and last $a$ can be omitted and $*$s represent generators other than $a$, using simple contractions;

\begin{gather}\label{equation:simplecontractions}
	a^2 = x^2 = b_\omega^2 = c_\omega^2 = d_\omega^2 = \ao{b}^2 = \ao{c}^2 = \ao{d}^2 = 1 \nonumber
	\\
	b_\omega c_\omega = c_\omega b_\omega = d_\omega, \quad
	c_\omega d_\omega = d_\omega c_\omega = b_\omega, \quad
	d_\omega b_\omega = b_\omega d_\omega = c_\omega \nonumber
	\\
	\ao{b} \ao{c} = \ao{c} \ao{b} = d_\omega, \quad
	\ao{c} \ao{d} = \ao{d} \ao{c} = b_\omega, \quad
	\ao{d} \ao{b} = \ao{b} \ao{d} = c_\omega \nonumber
	\\
	b_\omega \ao{c} = \ao{c} b_\omega = \ao{d}, \quad
	c_\omega \ao{d} = \ao{d} c_\omega = \ao{b}, \quad
	d_\omega \ao{b} = \ao{b} d_\omega = \ao{c} 
	\\
	\ao{b} c_\omega = c_\omega \ao{b} = \ao{d}, \quad
	\ao{c} d_\omega = d_\omega \ao{c} = \ao{b}, \quad
	\ao{d} b_\omega = b_\omega \ao{d} = \ao{c} \nonumber 
	\\
	b_\omega \ao{b} = \ao{b} b_\omega = c_\omega \ao{c} = \ao{c} c_\omega = d_\omega \ao{d} = \ao{d} d_\omega = x \nonumber
	\\
	b_\omega x = x b_\omega = \ao{b}, \quad
	c_\omega x = x c_\omega = \ao{c}, \quad
	d_\omega x = x d_\omega = \ao{d} \nonumber
	\\
	\ao{b} x = x \ao{b} = b_\omega, \quad
	\ao{c} x = x \ao{c} = c_\omega, \quad
	\ao{d} x = x \ao{d} = d_\omega \nonumber
\end{gather}

Denote $\ao{H} := \ao{H}^{(1)} := Stab_{\ao{G}}(1)$ and $g \in \ao{H}$ if and only if $g$ has even number of $a$'s. There is a natural embedding $\ao{\psi}$ from $\ao{H}$ into $\aso{G}\times \aso{G}$ given by $\ao{\psi}(g) = (g|_0,g|_1)$, where $g|_v$ is the restricted action on rooted tree with root $v$, for $v=0,1$. We will write $g=(g|_0,g_1)$ and omit subscript $\omega$ if there is no ambiguity.

for any group element (or a word) in $g \in Stab_{\ao{G}}(n)$, $g$ can be represented by $2^n$ elements $(g|_{i_1 i_2 \hdots i_n})_{i_1, i_2, \hdots, i_n \in \{0,1\}}$ by applying natural embedding $n$ times. This is called the decomposition of the group element (or the word) $g$ in to the depth of $n$, and can be visualize by a binary rooted tree with $n$ levels. If in some depth, all its elements $g|_{i_1 i_2 \hdots i_n}$ has length at most 1 (i.e. they are either the identity (empty word) or generators), then we call $(g|_{i_1 i_2 \hdots i_n})_{i_1, i_2, \hdots, i_n \in \{0,1\}}$ the nucleus of the element (word) $g$.

% Begin Subsection
\subsection{Space of Marked Groups} \label{marked}

The space of marked groups with $k$ generators $\mathcal{M}_k $, introduced in \cite{Gri84} is the space consisting of tuples $(G,S)$ where $S$ is an ordered set of $k$ elements generating the group $G$, together with the topology generated by the metric
\[ d((G_1,S_1),(G_2,S_2)) = 2^{-n} \]
where $n$ is the largest integer such that the balls of radius $n$ centered at identity of the Cayley graphs of $(G_1,S_1)$ and $(G_2,S_2)$ are identical.

Let $\{G_n\}$ be a sequence of groups in $\mathcal{M}_k$ and let $G \in \mathcal{M}_k$. We denote $G_n \Rightarrow G$ if $\{ G_n \}$ converges to $G$ under the metric topology of $\mathcal{M}_k$.

% Begin Section
\section{Modified Overgroups}

% Begin Subsection
\subsection{Algorithms for the Word Problem} \label{algo}

First we define the \textbf{algorithm} $\bm{\alpha}$ which solves the word problem for $\ao{G}$, when $\omega \in \Omega_0$. Given any reduced word $W$ of the alphabet $\{ a,b,c,d,x,\ttt{b} , \ttt{c}, \ttt{d} \}$, if it has even number of $a$'s, use natural embedding $\psi:\ttt{H} \rightarrow \ttt{G} \times \ttt{G}$ to get two reduced words $W_0,W_1$. If $W$ has odd number of $a$'s, terminate the algorithm. Similarly follow this process $N$ times, where $N = \lceil \log_2{|W|} \rceil$, to obtain $2^N$ reduced words $\{ W_{i_1 i_2 \hdots i_N} \}$, if such words exist. Then $|W_{i_1 i_2 \hdots i_N}| \leq \frac{|W|}{2^N} + 1 - \frac{1}{2^N}$ and so $|W_{i_1 i_2 \hdots i_N}|$ is either 0 or 1, and thus the nucleus is achieved. The algorithm $\alpha$ gives positive result if all words $W_{i_1 i_2 \hdots i_N}$ are empty word; That is the nucleus of $W$ consists of empty words.

Now for $\{i,j,k\} = \{0,1,2\}$ (we will use this notation of $i,j,k$ throughout  rest of the text.) we define \textbf{algorithm} $\bm{\beta_{ij}}$ which solves the word problem for $\ao{G}$, when $\omega \in \Omega_1$ and $i,j$ occur in $\omega$ infinitely often. Let $N_0$ be the largest index such that $\omega_{N_0} = k$. Given any reduced word $W$ of the alphabet $\{ a,b,c,d,x,\ttt{b} , \ttt{c}, \ttt{d} \}$, similarly to above, use natural embedding $\psi:\ttt{H} \rightarrow \ttt{G} \times \ttt{G}$ to get two reduced words $W_0,W_1$, if such words exist. Follow this process $N$ times, where $N = \max \{N_0, \lceil \log_2{|W|} \rceil \}$, to obtain $2^N$ reduced words $\{ W_{i_1 i_2 \hdots i_N} \}$, if such words exist. Then $|W_{i_1 i_2 \hdots i_N}| \leq \frac{|W|}{2^N} + 1 - \frac{1}{2^N}$ and so $|W_{i_1 i_2 \hdots i_N}|$ is either 0 or 1, and thus the nucleus is achieved. The algorithm gives positive result if all words $W_{i_1 i_2 \hdots i_N}$ are either empty word or $e_{ij}$, where $e_{01} = \ttt{b}, e_{12} = \ttt{d}$ and $e_{20} = \ttt{c}$; That is the nucleus of $W$ consists of empty words and $e_{ij}$'s.

% Begin Subsection
\subsection{Modified Overgroups} \label{modified}

Here we will introduce new collection of groups using the algorithms described above, named modified overgroups, similar to modified Grigorchuk groups $G_\omega^\alpha$ introduced in \cite{Gri84}. (The notation used in \cite{Gri84} is $\ttt{G}$, which is already taken to overgroups in this text.)

For $\omega \in \Omega$, define modified overgroup $\bm{\aoa{G}}$ as follows: $\ao{G}^\alpha$ is generated by eight elements $\aaa{a},\aaa{b_\omega},\aaa{c_\omega},\aaa{d_\omega},\aaa{x}, \aoa{b},\aoa{c}, \aoa{d}$ satisfying the simple contractions  (\ref{equation:simplecontractions}), and each reduced word ${W}$ represents the identity element in $\aoa{G}$ if and only if $W$ gives positive result when algorithm $\alpha$ is applied.

For $\omega \in \Omega_1 \cup \Omega_2$ with at most finitely many $k$'s, we define modified overgroups $\bm{\aob{G}}$ as follows: $\aob{G}$ is generated by eight elements $\bbb{a}, \bbb{b_\omega}, \bbb{c_\omega},$ $\bbb{d_\omega}, \bbb{x}, \aob{b}, \aob{c}, \aob{d}$ satisfying the simple contractions  (\ref{equation:simplecontractions}), and each reduced word ${W}$ represents the identity element in $\aob{G}$ if and only if $W$ gives positive result when algorithm $\beta_{ij}$ is applied.

\begin{proposition}\label{compare}
	If $\omega \in \Omega_0$, then $\aoa{G} = \ao{G}$ and if $\omega \in \Omega_1 \cup \Omega_2$, then $\aoa{G}$ surjects onto $\ao{G}$ with non trivial kernel.
	\\
	If $\omega \in \Omega_1$, then $\aob{G} = \ao{G}$ and if $\omega \in \Omega_2$, then $\aob{G}$ surjects onto $\ao{G}$ with non trivial kernel.
\end{proposition}

\begin{proof}
	Note that if $\omega \in \Omega_0$, then for any $n$, each element in $\as{G}{n}$ of length 1 will never the identity. Therefore, $W = 1$ in $\aoa{G} \iff W=1$ in $\ao{G}$, and so the modified overgroup $\aoa{G}$ is the same as generalized overgroup. If $\omega \in \Omega_1 \cup \Omega_2$, then for some $N$, $\sigma^N \omega$ contains at most two symbols. Say $\sigma^N \omega$ does not contain 2. Then $\as{b}{n} = 1$ in $\as{G}{n}$, but in modified overgroup $\as{G}{n}^\alpha$ it is not identity. Therefore $\ao{G} \neq \aoa{G}$. But any relation in $\aoa{G}$ is in fact a relation in $\ao{G}$ and therefore, $\aoa{G}$ surjects onto $\ao{G}$.
	
	If $\omega \in \Omega_1$ with finitely many $k$'s, then each element in $\as{G}{n}$ of length 1 will never the identity, unless it is $e_{ij}$. Therefore, $W = 1$ in $\aob{G} \iff W=1$ in $\ao{G}$, and so the modified overgroup $\aob{G}$ is the same as generalized overgroup. If $\omega \in \Omega_2$, then for some $N$, $\sigma^N \omega$ contains only one symbol. Say $\sigma^N \omega$ contain only 0's. Then $\as{c}{n} = 1$ in $\as{G}{n}$, but in modified overgroup $\as{G}{n}^{\beta_{01}}$ it is not identity. Therefore $\ao{G} \neq \aob{G}$. But any relation in $\aob{G}$ is in fact a relation in $\ao{G}$ and therefore, $\aob{G}$ surjects onto $\ao{G}$.
\end{proof}

The following proposition is useful in comparing two groups.

\begin{proposition}\label{oande}
	Let $r \in \mathbb{N}$ and let $\omega, \eta \in \Omega$ such that $\omega_i = \eta_i$ for each $i \leq N$, where $N > \log_2{(2r)}$.
	\begin{enumerate}
		\item If $\omega, \eta$ have all three symbols after the $N$-th position, then the balls of radius $r$ of Cayley graphs of $\ao{G}, \ttt{G}_\eta$ are identical.
		\item If $\omega$ has all three symbols after the $N$-th position, then the balls of radius $r$ of Cayley graphs of $\ao{G}, \ttt{G}^\alpha_\eta$ are identical.
		\item If $\omega, \eta$ have exactly the same two symbols, say $\{i,j\}$, after the $N$-th position, then the balls of radius $r$ of Cayley graphs of $\ao{G}, \ttt{G}_\eta$ are identical.
		\item If $\omega$ has only $i,j$ and $\eta$ has no $k$, after the $N$-th position, then the balls of radius $r$ of Cayley graphs of $\ao{G}, \ttt{G}^{\beta_{ij}}_\eta$ are identical.
	\end{enumerate}
\end{proposition}

\begin{proof}
	\quad (1) We will say two words $W,X$ of alphabets of generators of $\ao{G},\ttt{G}_\eta$, are equal if their corresponding letters match. Let $W,X$ be equal words of length at most $2r$. Suppose $W = 1$ in $\ao{G}$. Thus we can decompose $W$ into two words $\{W_0,W_1\}$, four words $\{W_{00},W_{01},W_{10},W_{11}\}$, $\hdots$, $2^N$ words $\{W_{i_1i_2 \hdots i_N}\}$, where all these words represents identity in corresponding groups. Note that $|W_{i_1i_2 \hdots i_N}| \leq \frac{|W|}{2^N} + 1 - \frac{1}{2^N} <2$. This together with the fact that $\omega$ has all three symbols implies that $W_{i_1i_2 \hdots i_N} = 1$ as a word. Also note that all the words $W_{i_1i_2 \hdots i_N}$ are described by first $N$ symbols of $\omega$ and since first $N$ symbols of $\omega$ and $\eta$ are equal, $X =1$ in $\ttt{G}_\eta$. Therefore we prove (1). \quad (2) The same argument as of (1) works, since no word of length 1 is identity in $\ttt{G}^\alpha_\eta$. \quad (3) Since $\omega, \eta$ only have $i,j$ after $N$-th position, only length 1 element which represent identity is $e_{ij}$. And therefore proof in (1) works. \quad (4)  The same argument as of (3) works, since only word of length 1 which is identity in $\ttt{G}^{\beta_{ij}}_\eta$ is $e_{ij}$.
\end{proof}

The modified overgroups behave nicely under limits.

\begin{corollary}\label{limit}
	Let $\{ \omega^{(n)} \}$ be a sequence in $\Omega$ and let $\omega \in \Omega$. Then,
	\[ \omega^{(n)} \rightarrow \omega \implies \left( \ttt{G}_{\omega^{(n)}}^\alpha \Rightarrow \aoa{G} \right). \]
	If there is an $N$ such that no $k$ appear after the $N$-th position of each of $\{ \omega^{(n)} \}$, then
	\[ \omega^{(n)} \rightarrow \omega \implies \left( \ttt{G}_{\omega^{(n)}}^{\beta_{ij}} \Rightarrow \aob{G} \right). \]
\end{corollary}

\begin{proof}
	Since $\omega^{(n)} \rightarrow \omega$, we can pick an $\omega^{(n)}$ satisfying hypothesis in proposition \ref{oande}, and using proposition \ref{oande} we get balls of radius $k$ of Cayley graphs of $\ttt{G}_{\omega^{(n)}}^\alpha$ and $\aoa{G}$, are identical. Therefore, $\ttt{G}_{\omega^{(n)}}^\alpha \Rightarrow \aoa{G}$.
	
	Using a similar argument we get $\ttt{G}_{\omega^{(n)}}^{\beta_{ij}} \Rightarrow \aob{G}$, under the hypothesis given.	
\end{proof}

% Begin Subsection
\subsection{Modified overgroups for some ${\omega \in \Omega}$}

Now we will look at the modified overgroups and see what their structures are. In fact we will prove theorem \ref{modified-main} using propositions introduced in this section. First we will introduce some words and substitution rules which will be used throughout this section.

Let $y \neq 1$ be a generator (a group element with length 1) of modified overgroup such that for each $n \in \mathbb{N}$ the decomposition of $y$ into depth $n$ has nucleus $(1,1,\hdots, 1,y)$; That is $y$ at $1^n$-th position and $1$ (or the empty word) every other position. For each $n \in \mathbb{Z}$, define $v_n(y) = v_n$ by
\begin{equation}\label{equation:vn}
	v_n =
	\begin{cases}
		y^{(ax)^n} & ; n \geq 0 \\
		y^{(ax)^{-n-1}a} & ; n <0
	\end{cases}.
\end{equation}
Then we can observe the following properties:

\begin{proposition}\label{vn_proporties}
	\begin{enumerate}
		\item $v_n \neq 1, v_n^2 = 1$ and $v_n^a = v_{-n-1}$ for each $n \in \mathbb{Z}$.
		\item $v_n = (1,v_{n/2})$ when $n$ is even and $v_n = (v_{-(n+1)/2},1)$ when $n$ is odd.
		\item $v_n$'s are mutually distinct and mutually commutative for all $n$.
		\item $ax$ acts on $v_n$ by conjugation and $v_n^{(ax)} = v_{n+1}$.
	\end{enumerate}
\end{proposition}

\begin{proof}
	$v_n^2 =1$ since $y$ is an involution. Note that $y = (1,y)$ by natural embedding. Then direct calculation yields (2) and $v_n^a = v_{-n-1}$. Using (2) and induction on $n$, we can show that $v_n \neq 1$ and $v_n$'s are mutually commutative. Now it is straight forward to show $v_n$'s are mutually distinct by (2) and (1). (4) is also by direct calculation.
\end{proof}

Now we will introduce two substitution rules $\xi_0, \xi_1$:
\begin{equation}\label{equation:substitution rule}
	\xi_0 =
	\begin{cases}
		a \mapsto x\\
		x \mapsto axa\\
		y \mapsto aya
	\end{cases}
	\hspace{1cm}
	\xi_1 =
	\begin{cases}
		a \mapsto axa\\
		x \mapsto x\\
		y \mapsto y
	\end{cases}
\end{equation}
Note that $\xi_1((ax)^n) = (ax)^{2n}, \xi_1((ax)^na) = (ax)^{2n+1}a, \xi_0((ax)^n) = a(ax)^{2n}a$ and $\xi_0((ax)^na) = a(ax)^{2n+1}$. Then $\xi_1(v_n) = v_{2n} = (1,v_n)$ and $\xi_0(v_n) = v_{-2n-1} = (v_n,1)$. Now we will recursively construct words $V(y)_{i_1i_2 \hdots i_n} = V_{i_1i_2 \hdots i_n}$'s, corresponding to vertices $i_1i_2 \hdots i_n$ of $T_2$, as follows (see figure \ref{fig:tree V values}):
\[V_\emptyset = v_0\]
\begin{equation}\label{equation:V words}
	V_{i_1i_2 \hdots i_n} = \xi_{i_1}(V_{i_2 \hdots i_n})
\end{equation}

% Node styles
\tikzset{
	% Two node styles for game trees: solid and hollow
	s/.style={circle,draw,inner sep=0,fill=black},
	l/.style={circle,draw,inner sep=0,fill=black,xshift=10},
	r/.style={circle,draw,inner sep=0,fill=black,xshift=-10},
	%    t/.style={regular polygon, regular polygon sides=3,draw,inner sep=2,yshift=3}
}
\begin{figure}[t!]
	\centering
	\begin{tikzpicture}[scale=1,font=\footnotesize]
	% Specify spacing for each level of the tree
	\tikzstyle{level 1}=[level distance=10mm,sibling distance=60mm]
	\tikzstyle{level 2}=[level distance=15mm,sibling distance=30mm]
	\tikzstyle{level 3}=[level distance=15mm,sibling distance=15mm]
	\tikzstyle{level 4}=[level distance=10mm,sibling distance=10mm]
	% The Tree
	\node(-1)[s,label=above:{$v_0$}]{}
	child{node(0)[s,label=above:{$v_{-1}$}]{}
		child{node(00)[s,label=left:{${v_1}$}]{}
			child{node(100)[s,label=below:{${v_{-3}}$}]{}}
			child{node(101)[s,label=below:{${v_1}$}]{}}
		}
		child{node(01)[s,label=right:{${v_{-1}}$}]{}
			child{node(110)[s,label=below:{${v_3}$}]{}}
			child{node(111)[s,label=below:{${v_{-1}}$}]{}}
		}
	}
	child{node(1)[s,label=above:{$v_0$}]{}
		child{node(10)[s,label=left:{${v_{-2}}$}]{}
			child{node(100)[s,label=below:{${v_2}$}]{}}
			child{node(101)[s,label=below:{${v_{-2}}$}]{}}
		}
		child{node(11)[s,label=right:{${v_0}$}]{}
			child{node(110)[s,label=below:{${v_{-4}}$}]{}}
			child{node(111)[s,label=below    :{${v_0}$}]{}}
		}
	};
	\end{tikzpicture}
	\caption{$V_{i_1i_2 \hdots i_n}$ values of first 3 levels}
	\label{fig:tree V values}
\end{figure}
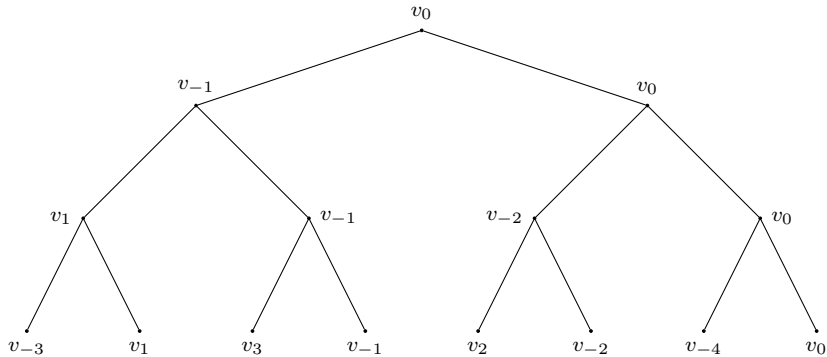

It is easy to see that $V_{i_1i_2 \hdots i_n} = v_k$ for some $k \in \mathbb{Z}$ and has a nucleus of depth $n$ with $y$ in $i_1i_2 \hdots i_n$-th coordinate and empty word in other coordinates. Now we will introduce some propositions, which describes the group structure of modified groups for $\omega = 0^\infty$ and $\omega \in \{0,1\}^\mathbb{N}$.

\begin{proposition}\label{alpha-gp-omega2}
	$\at{G}{0^\infty}^\alpha$ is virtually $\mathcal{L}_3$ of index 2.
\end{proposition}
\begin{proof}
	Let $\ttt{G} := \at{G}{0^\infty}^\alpha = \left< a,b,c,d,x,\ttt{b},\ttt{c},\ttt{d} \right>$ and let $G := \at{G}{0^\infty}$. We will drop the subscript $0^\infty$ and superscript $\alpha$ for the convenience. Also we will use the same letters for generating sets of $\ttt{G}$ and $G$ since there will be no ambiguity. Note that in $G$ we have $b = c = \ttt{d} = x$ and $d = \ttt{b} = \ttt{c} =1$. Therefore $G$ is isomorphic to the infinite dihedral group $D_\infty$ generated by $a$ and $b$. Let $\phi$ be the surjection from $\ttt{G}$ to $G$ described in proposition \ref{compare}.
	
	\begin{lemma}\label{word}
		$Ker(\phi) = \left< \left< d,\ttt{b},\ttt{c}\right> \right> = \left< v(d)_n, v(\ttt{b})_n, v(\ttt{c})_n | n \in \mathbb{Z} \right> \cong \bigoplus_{\mathbb{Z}}{\mathbb{Z}_2^3}$. Here $\left<\left<\cdot\right>\right>$ represents the normal closure.
		\begin{proof}
			The inclusion $\left< v(d)_n, v(\ttt{b})_n, v(\ttt{c})_n | n \in \mathbb{Z} \right> \leq \left< \left< d,\ttt{b},\ttt{c}\right> \right> \leq Ker(\phi)$ is trivial. To show the other inclusion, let $g \in Ker(\phi)$ and let $W$ be a reduced word representing $g$ in $\ttt{G}$. Since $g \in Ker(\phi)$, $W=1$ in $G$. But a word is the identity in $G$ if and only if its nucleus contains only $1, d, \ttt{b}, \ttt{c}$. So $W$ has nucleus with only $1, d, \ttt{b}, \ttt{c}$. We can construct a word $W'$ using $V(d)_{i_1i_2 \hdots i_n}, V(\ttt{b})_{i_1i_2 \hdots i_n}$ and $V(\ttt{c})_{i_1i_2 \hdots i_n}$ so that the nucleus of $W'$ is the same as the nucleus of $W$. Thus $g = W = W'$ and so $Ker(\phi) \leq \left< v(d)_n, v(\ttt{b})_n, v(\ttt{c})_n | n \in \mathbb{Z} \right>$. Therefore we get the equality of three groups. Using a similar argument as of the proof of proposition \ref{vn_proporties} (3), we see that $v(d)_n, v(\ttt{b})_n, v(\ttt{c})_n$ are distinct and therefore by proposition \ref{vn_proporties}, we get $\left< v(d)_n, v(\ttt{b})_n, v(\ttt{c})_n | n \in \mathbb{Z} \right> \cong \bigoplus_{\mathbb{Z}}{\mathbb{Z}_2^3}$. This completes the proof of lemma.
		\end{proof}
	\end{lemma}
	
	Note that the generator of $\left<ax\right>$ acts on $Ker(\phi)$ by shifting its generators. Also note that $Ker(\phi)$ and $\left<ax\right>$ intersects trivially, since $ax$ is of infinite order and all elements of $Ker(\phi)$ are involutions. So, $Ker(\phi) \rtimes \left<ax\right>$ is isomorphic to $\mathcal{L}_3 = \mathbb{Z}_2^3 \wr \mathbb{Z}$.
	
	Conjugating the generators of $Ker(\phi) \rtimes \left<ax\right>$ by generators of $\ttt{G}$, we see that $Ker(\phi) \rtimes \left<ab\right>$ is normal in $\ttt{G}$. The quotient $\ttt{G}/Ker(\phi) \cong D_\infty$ maps onto the quotient $\ttt{G}/\left(Ker(\phi) \rtimes \left<ax\right>\right)$. The kernel of the homomorphism from  $\ttt{G}/Ker(\phi)$ to $\ttt{G}/\left(Ker(\phi) \rtimes \left<ax\right>\right)$ is generated by the image of $ax$ in $\ttt{G}/Ker(\phi)$. So $Ker(\phi) \rtimes \left<ax\right>$ has index 2 in $\ttt{G}$, and therefore $\ttt{G}$ is almost $Ker(\phi) \rtimes \left<ax\right> \cong \mathcal{L}_3$ with index 2.
\end{proof}

\begin{proposition}\label{beta-gp-omega2}
	$\at{G}{0^\infty}^{\beta_{ij}}$ is virtually $\mathcal{L}$ with index 2. 
\end{proposition}

\begin{proof}
	To simplicity, consider $\at{G}{0^\infty}^{\beta_{01}}$. we will show $\at{G}{0^\infty}^{\beta_{01}} \cong {G}_{0^\infty}^{\alpha}$. Let $f:\at{G}{0^\infty}^{\beta_{01}} \rightarrow {G}_{0^\infty}^{\alpha}$ be defined by,
	\[f: \begin{cases} 
	a^{\beta_{01}} & \mapsto \quad a^{\alpha} \\
	b_{0^\infty}^{\beta_{01}} & \mapsto \quad b_{0^\infty}^{\alpha} \\
	c_{0^\infty}^{\beta_{01}} & \mapsto \quad c_{0^\infty}^{\alpha} \\
	d_{0^\infty}^{\beta_{01}} & \mapsto \quad d_{0^\infty}^{\alpha} \\
	x^{\beta_{01}} & \mapsto \quad b_{0^\infty}^{\alpha} \\
	\ttt{b}_{0^\infty}^{\beta_{01}} & \mapsto \quad 1 \\
	\ttt{c}_{0^\infty}^{\beta_{01}} & \mapsto \quad d_{0^\infty}^{\alpha} \\
	\ttt{d}_{0^\infty}^{\beta_{01}} & \mapsto \quad c_{0^\infty}^{\alpha} 
	\end{cases}
	\] 
	Using algorithms $\alpha$ and $\beta_{01}$, we can show that $f$ is well defined group isomorphism. ${G}_{0^\infty}^{\alpha}$ is virtually $\mathcal{L}$ with index $2$ by \cite{BG14}.
\end{proof} 

\begin{proposition}\label{alpha-gp-omega1}
	Let $\omega \in \{0,1\}^\mathbb{N}$. Then $\ao{G}^\alpha$ contains $\mathcal{L}$ as a subgroup and is an extension of $\ao{G}$ by $\bigoplus_{\mathbb{Z}}{\mathbb{Z}_2}$ and 
\end{proposition}

\begin{proof}
	Let $\omega \in \{0,1\}^\mathbb{N}$. Let $\ttt{G} := \ao{G}^\alpha = \left< a,b,c,d,x,\ttt{b},\ttt{c},\ttt{d} \right>$ and let $G := \ao{G}$. We will drop the subscript $\omega$ and superscript $\alpha$ for the convenience. Also we will use the same letters for generating sets of $\ttt{G}$ and $G$ since there will be no ambiguity. Note that in $G$ we have $b = x$ and $\ttt{b} =1$. Let $\phi$ be the surjection from $\ttt{G}$ to $G$ described in proposition \ref{compare}. Then by a similar argument as of the proof of lemma \ref{word}, $Ker(\phi) = \left< \left< \ttt{b}\right> \right> = \left< v(\ttt{b})_n | n \in \mathbb{Z} \right> \cong \bigoplus_{\mathbb{Z}}{\mathbb{Z}_2}$. Hence $\ttt{G}$ is an extension of $G$ by $\bigoplus_{\mathbb{Z}}{\mathbb{Z}_2}$. Also since $Ker(\phi) \cap \left< ax \right> = \left< 1 \right>$ and $ax$ acts on $Ker(\phi)$ by shifting, $Ker(\phi) \rtimes \left< ax \right> \cong \mathcal{L}$ is a subgroups of $\ttt{G}$.
\end{proof}

\begin{proof}[Proof of Theorem \ref{modified-main}]
	Note that for any $\omega \in \Omega$, $\ao{G}$ is commensurable to $(\as{G}{N})^{2^N}$. This together with propositions \ref{alpha-gp-omega2}, \ref{beta-gp-omega2} and \ref{alpha-gp-omega1} proves the result.
\end{proof}

% Begin Section
\section{$\mathbf{\ao{G} \in \mathcal{M}_8}$} \label{ydescription}

Recall the notation introduced in equation (\ref{equation:sets}).

\begin{align}
\mathcal{X}  = &\{ (\ao{G},\ao{S}) \}_{\omega \in \Omega} \nonumber \\ 
\mathcal{X}_i  = &\{ (\ao{G},\ao{S}) \}_{\omega \in \Omega_i} \text{ ; for } i = 0, 1, 2 \nonumber\\
\mathcal{X}_i^\alpha = &\{ (\ao{G},\ao{S}) \}_{\omega \in \Omega_i} \text{ ; for } i = 1, 2 \nonumber\\
\mathcal{X}_2^\beta  = &\{ (\ao{G}^\beta,\ao{S}^\beta) \; | \; \beta \in \{ \beta_{01}, \beta_{12}, \beta_{20} \} \; \}_{\omega \in \Omega_2}  \nonumber\\
\mathcal{Y}  = &\mathcal{X}_0 \cup \mathcal{X}_1^\alpha \cup \mathcal{X}_2^\alpha   \nonumber
\end{align}
\\
Then $\mathcal{X}$ is the disjoint union of $\mathcal{X}_0, \mathcal{X}_1, \mathcal{X}_2$. In order to prove the theorem \ref{distinct-main} we have following propositions.

\begin{proposition}\label{different-omega}
	Generalized overgroups and modified overgroups corresponding to different oracles $\omega$, are different. 
\end{proposition}

\begin{proof}
	Note that the classical Grigorchuk's groups and their modifications are embedded in generalized overgroups and modified overgroups. By \cite{Gri84}, different oracles $\omega$ give rise to different classical Grigorchuk's groups and their modifications. Therefore by extending the generation set, we get the result.  
\end{proof}

Form the above proposition we can see that the sets $\mathcal{X}_0, (\mathcal{X}_1 \cup \aaa{\mathcal{X}}_1), (\mathcal{X}_2 \cup \aaa{\mathcal{X}}_2 \cup \bbbb{\mathcal{X}}_2)$ are disjoint. This together with corollary \ref{growth-main}, we get

\begin{corollary}\label{distinct-middle}
	$\mathcal{X}_0, \mathcal{X}_1, \mathcal{X}_2, \aaa{\mathcal{X}}_1, (\aaa{\mathcal{X}}_2 \cup \bbbb{\mathcal{X}}_2)$ are disjoint.
\end{corollary} 

Now let us prove $\aaa{\mathcal{X}}_2, \bbbb{\mathcal{X}}_2$ are disjoint.

\begin{proposition}\label{QR}
	$\aaa{\mathcal{X}}_2, \bbbb{\mathcal{X}}_2$ are disjoint. In fact for $\omega \in \Omega_2$ with infinitely many $i$'s, the groups $\ao{G}^\alpha, \ttt{G}_\omega^{\beta_{ij}}$ and $\ttt{G}_\omega^{\beta_{ik}}$ are different. 
\end{proposition}

In order to prove this, we will construct word $W(ij)$ such that it's nucleus consists only of $1$ and $e_{ij}$, with not all $1$'s. For ease of writing let us assume $\omega \in \Omega_2$ with infinitely many 0's. We will construct the word $W(01)$. Recall that $e_{01}=\ttt{b}$. Let $\omega = \omega_1 \omega_2 \hdots \omega_n 2^k \eta_1 \eta_2 \hdots \eta_r 0^\infty$, where $\omega_n \neq 2, \eta_r \neq 0$ and $\eta_i \neq 2$ for all $i$. Now define $X_i, Y_i, Z_i$ for $i = 0,1,\hdots, n$ as follows;
\\
$X_i = \ttt{b}$ if $\omega_i = 2$ and $X_i = b$ if $\omega_i \neq 2$, $Y_i = X_n^{X_{n-1}^{\hdots X_i^a}}$, and $Z_i = (\ttt{b}Y_i)^2$

Now consider the word $W=W(01) = (Z_1)^{2^k}$. Then the decomposed diagram of $W$ of depth $n+k$ is given in the figure \ref{fig:decomposition} and thus it's nucleus consists of only $1,\ttt{b}$.

Using similar constructions, we can construct words $W(02), W(12)$.

% Node styles
\tikzset{
	% Two node styles for game trees: solid and hollow
	s/.style={circle,draw,inner sep=0,fill=black},
	l/.style={circle,draw,inner sep=0,fill=black,xshift=10},
	rb/.style={circle,draw=none,inner sep=0,fill=black},
	%    t/.style={regular polygon, regular polygon sides=3,draw,inner sep=2,yshift=3}
}
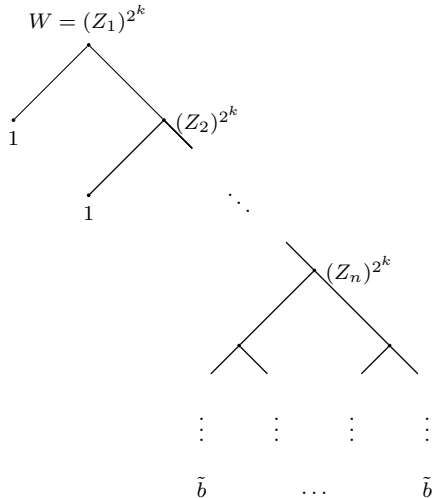
\begin{figure}[t!]
	\centering
	\begin{tikzpicture}[scale=1,font=\footnotesize]
	% Specify spacing for each level of the tree
	\tikzstyle{level 1}=[level distance=10mm,sibling distance=20mm]
	\tikzstyle{level 2}=[level distance=10mm,sibling distance=20mm]
	\tikzstyle{level 3}=[level distance=10mm,sibling distance=25mm]
	\tikzstyle{level 4}=[level distance=10mm,sibling distance=20mm]
	% The Tree
	\node(-1)[s,label=above:{$W = (Z_1)^{2^k}$}]{}
	child{node(0)[s,label=below:{$1$}]{}
	}
	child{node(1)[s,label=right:{$(Z_2)^{2^k}$}]{}
	};
	\node(10)[circle,draw,inner sep=0,fill=black, xshift=-10mm, yshift=-10mm,label=below:{$1$}] at ($(1)$) {};
	\draw[line width=0.5 pt] (1) to (10);
	\node(1-11)[xshift=5mm, yshift=-5mm,label=center:{}] at ($(1)$) {};
	\draw[line width=0.5 pt] (1) to (1-11);
	\node(1--11)[xshift=10mm, yshift=-10mm,label=center:{$\ddots$}] at ($(1)$) {};
	\draw[line width=0.5 pt] (1) to (1-11);
	\node(1---11)[xshift=15mm, yshift=-15mm,label=center:{}] at ($(1)$) {};
	\draw[line width=0.5 pt] (1) to (1-11);
	\node(11)[circle,draw,inner sep=0,fill=black, xshift=20mm, yshift=-20mm,label=right:{$(Z_n)^{2^k}$}] at ($(1)$) {};
	\draw[line width=0.5 pt] (1---11) to (11);
	
	\node(110)[circle,draw,inner sep=0,fill=black, xshift=-10mm, yshift=-10mm,label=right:{}] at ($(11)$) {};
	\node(111)[circle,draw,inner sep=0,fill=black, xshift=10mm, yshift=-10mm,label=right:{}] at ($(11)$) {};
	\draw[line width=0.5 pt] (11) to (110);
	\draw[line width=0.5 pt] (11) to (111);	
	
	\node(1100)[xshift=-5mm, yshift=-5mm,label=below:{$\vdots$}] at ($(110)$) {};
	\draw[line width=0.5 pt] (110) to (1100);
	\node(1101)[xshift=5mm, yshift=-5mm,label=below:{$\vdots$}] at ($(110)$) {};
	\draw[line width=0.5 pt] (110) to (1101);

	\node(1110)[xshift=-5mm, yshift=-5mm,label=below:{$\vdots$}] at ($(111)$) {};
	\draw[line width=0.5 pt] (111) to (1110);
	\node(1111)[xshift=5mm, yshift=-5mm,label=below:{$\vdots$}] at ($(111)$) {};
	\draw[line width=0.5 pt] (111) to (1111);
	
	\node(11100)[xshift=0mm, yshift=-10mm,label=below:{$\ttt{b}$}] at ($(1100)$) {};
	\node(11111)[xshift=0mm, yshift=-10mm,label=below:{$\ttt{b}$}] at ($(1111)$) {};
	\node(11011)[xshift=0mm, yshift=-27mm,label=below:{$\hdots$}] at ($(11)$) {};
	\end{tikzpicture}
	\caption{Decomposition of $W = W(01)$ in to the depth $n+k$}
	\label{fig:decomposition}
\end{figure}

\begin{proof}[Proof of Proposition \ref{QR}]
	Suppose $\omega = \omega_1 \omega_2 \hdots \omega_n 2^k \eta_1 \eta_2 \hdots \eta_r 0^\infty$, where $\omega_n \neq 2, \eta_r \neq 0$ and $\eta_i \neq 2$ for all $i$ and consider the word $W = W(01)$ defined as above. Then $W$ represent the identity element in $\ttt{G}_\omega^{\beta_{01}}$ but not the identity in $\ao{G}^\alpha$ and $\ttt{G}_\omega^{\beta_{02}}$. Similarly using the word $W(02)$, we can show $\ao{G}^\alpha \neq \ttt{G}_\omega^{\beta_{02}}$.
\end{proof}

\begin{proof}[Proof of Theorem \ref{distinct-main}]
	Directly from proposition \ref{different-omega}, \ref{QR} and corollary \ref{distinct-middle}.
\end{proof}

Now we will prove theorem \ref{closure}. We will use few lemmas in order to do this.

\begin{lemma}\label{lim}
	Let $\omega, \omega^{(n)} \in \Omega$ for all $n \in \mathbb{N}$ and $\omega^{(n)} \rightarrow \omega$. Suppose $G = \lim {\ttt{G}_{\omega^{(n)}}}$ exists. Then $G = \ao{G}, \aoa{G}$ or $\aob{G}$. Moreover $G \in \mathcal{Y} \cup \mathcal{X}_1 \cup \bbbb{\mathcal{X}}_2$ and so $G \notin \mathcal{X}_2$.
	\begin{proof}
		If $\omega \in \Omega_0$, Since $\omega^{(n)} \rightarrow \omega$, using proposition \ref{oande} (1), we get $G = \ao{G}$. Now consider $\omega \in \Omega_1$; Let $N$ be the smallest index such that only two symbols appear after $N$-th position. If there are infinitely many $\omega^{(n)}$'s with all three symbols appearing after $N$-th position, then by proposition \ref{oande} (2), $G = \aoa{G}$. If all but finitely many $\omega^{(n)}$'s contain only two symbols after the $N$-th position, then by proposition \ref{oande} (3), $G = \ao{G}$.
		
		Finally consider $\omega \in \Omega_2$; Let $N$ be the smallest index such that only one symbol appear after the $N$-the position. If there are infinitely many $\omega^{(n)}$'s with all three symbols after the $N$-th position, then by proposition \ref{oande} (2), $G = \aoa{G}$. If all but finitely many $\omega^{(n)}$'s contain only two symbols, say $\{ i,j \}$, after the $N$-th position, then by proposition \ref{oande} (4), $G = \aob{G}$.
		
		From above, we can conclude that $G \in \mathcal{Y} \cup \mathcal{X}_1 \cup \bbbb{\mathcal{X}}_2$ and $G \notin \mathcal{X}_2$.
	\end{proof}
\end{lemma}

\begin{proof}[Proof of Theorem \ref{closure} (1)]
	To the contrary, suppose there is an $\eta \in \Omega_2$ such that $\ttt{G}_\eta \in \mathcal{X}_2$ is a limit point. Then there exists a sequence $\{ G_{\omega^{(n)}} \}$ converging to $\ttt{G}_\eta$. Since $\Omega$ is compact, we may assume $\omega^{(n)} \rightarrow \omega$, for some $\omega \in \Omega$. By lemma \ref{lim}, $\ttt{G}_\eta = \lim \ttt{G}_{\omega^{(n)}} \notin \mathcal{X}_2$, which is a contradiction.
\end{proof}

\begin{proof}[Proof of Theorem \ref{closure} (3)(a)]
	Let $G \in \mathcal{Y}_\sharp = (\mathcal{X}_0 \cup \aaa{\mathcal{X}}_1 \cup \aaa{\mathcal{X}}_2)_\sharp$. Then there exists $\{\omega^{(n)}\} \subset \Omega$ such that $\ttt{G}_{\omega^{(n)}}^\alpha \Rightarrow G$. By compactness of $\Omega$ we may assume $\omega^{(n)} \rightarrow \omega$ for some $\omega \in \Omega$. Then $G = \ao{G}^\alpha$ by corollary \ref{limit}. This together with corollary \ref{limit} implies that \[ \omega^{(n)} \rightarrow \omega \iff \left( \ttt{G}_{\omega^{(n)}}^\alpha \Rightarrow \aoa{G} \right). \]
	Therefore $\mathcal{Y} \cong \Omega$ and $\mathcal{X}_0 \cong \Omega_0, \aaa{\mathcal{X}}_1 \cong \Omega_1, \aaa{\mathcal{X}}_2 \cong \Omega_2$. Thus, $\mathcal{Y}$ is homeomorphic to a Cantor set and $\mathcal{Y} = (\mathcal{X}_0)_\sharp = (\aaa{\mathcal{X}}_1)_\sharp = (\aaa{\mathcal{X}}_2)_\sharp$.
\end{proof}

\begin{proof}[Proof of Theorem \ref{closure} (3)(b)]
	First we will show $\mathcal{X}_\sharp \subset \mathcal{Y} \cup \mathcal{X}_1 \cup \bbbb{\mathcal{X}}_2$. Let $G$ be a limit point of $\mathcal{X}$. Thus there exists a sequence $\{ G_{\omega^{(n)}} \}$ converging to $G$. Since $\Omega$ is compact, we may assume $\omega^{(n)} \rightarrow \omega$, for some $\omega \in \Omega$. Then by lemma \ref{lim}, $G \in \mathcal{Y} \cup \mathcal{X}_1 \cup \bbbb{\mathcal{X}}_2$. Therefore $\mathcal{X}_\sharp \subset \mathcal{Y} \cup \mathcal{X}_1 \cup \bbbb{\mathcal{X}}_2$.
	
	Now we will show $\mathcal{Y} \cup \mathcal{X}_1 \cup \bbbb{\mathcal{X}}_2 \subset (\mathcal{X}_1)_\sharp$. Let $\omega \in \Omega$ and choose $\omega^{(n)} = \omega_1 \omega_2 \hdots \omega_n (012) (ij)^\infty$, for each $n$. Then using proposition \ref{oande} (2), we get $\ttt{G}_{\omega^{(n)}} \Rightarrow \ao{G}^\alpha$. So $\mathcal{Y} \subset (\mathcal{X}_1)_\sharp$. Let $\omega \in \Omega_1 \cup \Omega_2$ with finitely many $k$'s. Choose $\omega^{(n)} = \omega_1 \omega_2 \hdots \omega_n (ij)^\infty$, for each $n$. Using proposition \ref{oande} (4), we get $\ttt{G}_{\omega^{(n)}} \Rightarrow \ao{G}^{\beta_{ij}}$. So $\mathcal{X}_1 \cup \bbbb{\mathcal{X}}_2 \subset (\mathcal{X}_1)_\sharp$. Therefore $\mathcal{Y} \cup \mathcal{X}_1 \cup \bbbb{\mathcal{X}}_2 \subset (\mathcal{X}_1)_\sharp$.
	
	Using a similar argument by choosing $\omega^{(n)} = \omega_1 \omega_2 \hdots \omega_n (012) (i)^\infty$ and again choosing $\omega^{(n)} = \omega_1 \omega_2 \hdots \omega_n (ij) (i)^\infty$, we can show $\mathcal{Y} \cup \mathcal{X}_1 \cup \bbbb{\mathcal{X}}_2 \subset (\mathcal{X}_2)_\sharp$.
	
	Since $\mathcal{X}_1$ and $\mathcal{X}_2$ are subsets of $\mathcal{X}$, we get  $\mathcal{X}_\sharp = (\mathcal{X}_1)_\sharp = (\mathcal{X}_2)_\sharp  = \mathcal{Y} \cup \mathcal{X}_1 \cup \bbbb{\mathcal{X}}_2$. Corollary \ref{limit} together with proposition \ref{compare} implies that $(\mathcal{X}_1)_\sharp = (\bbbb{\mathcal{X}}_2)_\sharp$ and so we get the desired result.
\end{proof}

Now we will complete the proof of theorem \ref{closure}.

\begin{proof}[Proof of theorem \ref{closure} (2)]
	We already proved $\mathcal{Y}$ is homeomorphic to a Cantor set. Now let us prove that $\mathcal{X}_\sharp$ is also homeomorphic to a Cantor set. Note that the set $\mathcal{X}_\sharp$ is a perfect set. (That is a closed set with all its point being limit points). The space $\mathcal{M}_8$ is a totally disconnected compact metric space. Let us recall that any non empty, totally disconnected, compact, perfect metric space is homeomorphic to the Cantor set. Therefore, $\mathcal{X}_\sharp$ homeomorphic to the Cantor set.
\end{proof}

% Bib

\end{document}